\documentclass[11pt,letterpaper]{amsart}
\usepackage{amsmath,amssymb,amsthm}
\usepackage{enumerate}

\DeclareMathOperator{\Tr}{Tr}
\DeclareMathOperator{\val}{val}
\DeclareMathOperator{\wt}{wt}
\newcommand{\C}{{\mathbb C}}
\newcommand{\F}{{\mathbb F}}
\newcommand{\Q}{{\mathbb Q}}
\newcommand{\Z}{{\mathbb Z}}
\newcommand{\Fp}{{\mathbb \F_p}}
\newcommand{\card}[1]{\left|{#1}\right|}
\newcommand{\mins}[1]{\min_{\substack{#1}}}
\newcommand{\sums}[1]{\sum_{\substack{#1}}}
\newcommand{\ceil}[1]{\lceil{#1}\rceil}
\newcommand{\grmul}[1]{{#1}^\times}
\newcommand{\mchars}[1]{{\widehat{\grmul{#1}}}}
\newcommand{\valp}{\val_p}
\newcommand{\mywalsh}[2]{S_{{#1},{#2}}}
\newtheorem{theorem}{Theorem}[section]
\newtheorem{proposition}[theorem]{Proposition}
\newtheorem{lemma}[theorem]{Lemma}
\newtheorem{corollary}[theorem]{Corollary}
\newtheorem{conjecture}[theorem]{Conjecture}
\theoremstyle{remark}
\newtheorem{remark}[theorem]{Remark}
\title[The $p$-Adic Valuations of Weil Sums of Binomials]{The $p$-Adic Valuations \\ of Weil Sums of Binomials}

\author[Katz]{Daniel J. Katz}
\address{Department of Mathematics, California State University, Northridge, \: United States}\thanks{The work of Katz, Lee, and Sapozhnikov on this paper was supported in part by the National Science Foundation under Grant DMS 1500856.}

\author[Langevin]{Philippe Langevin}
\address{Institut de Math\'ematiques de Toulon, Universit\'e de Toulon, France}

\author[Lee]{Sangman Lee}
\address{Department of Mathematics, California State University, Northridge, \: United States}

\author[Sapozhnikov]{Yakov Sapozhnikov}
\address{Department of Mathematics, California State University, Northridge, \: United States}

\date{first version: 13 August 2016; this version: 20 March 2017.}
\begin{document}
\begin{abstract}
We investigate the $p$-adic valuation of Weil sums of the form $W_{F,d}(a)=\sum_{x \in F} \psi(x^d -a x)$, where $F$ is a finite field of characteristic $p$, $\psi$ is the canonical additive character of $F$, the exponent $d$ is relatively prime to $|F^\times|$, and $a$ is an element of $F$.
Such sums often arise in arithmetical calculations and also have applications in information theory.
For each $F$ and $d$ one would like to know $V_{F,d}$, the minimum $p$-adic valuation of $W_{F,d}(a)$ as $a$ runs through the elements of $F$.
We exclude exponents $d$ that are congruent to a power of $p$ modulo $|F^\times|$ (degenerate $d$), which yield trivial Weil sums.
We prove that $V_{F,d} \leq (2/3)[F\colon{\mathbb F}_p]$ for any $F$ and any nondegenerate $d$, and prove that this bound is actually reached in infinitely many fields $F$.
We also prove some stronger bounds that apply when $[F\colon{\mathbb F}_p]$ is a power of $2$ or when $d$ is not congruent to $1$ modulo $p-1$, and show that each of these bounds is reached for infinitely many $F$.
\end{abstract}
\maketitle
\section{Introduction}
We consider Weil sums of binomials of the form
\begin{equation}\label{Alice}
W_{F,d}(a) = \sum_{x \in F} \psi(x^d-a x),
\end{equation}
where $F$ is a finite field, the exponent $d$ is a positive integer such that $\gcd(d,q-1)=1$, the coefficient $a$ is in $F$, and $\psi\colon F \to \C$ is the canonical additive character of $F$.
Throughout this paper, $F$ is always a finite field of characteristic $p$ and order $q=p^n$.
Then $\psi(x)=e^{2\pi i \Tr(x) /p}$, where the absolute trace  $\Tr \colon F\to\F_p$ is given by $\Tr(x)=x+x^p+\cdots+x^{p^{n-1}}$.
The condition on $d$ makes $x\mapsto x^d$ a permutation of $F$, which means that $W_{F,d}(0)=0$.
Every character sum of the more general form
\[
\sum_{x \in F} \psi(b x^s + c x^t)
\]
with $b \in \grmul{F}$, $c \in F$, and $\gcd(s,q-1)=\gcd(t,q-1)=1$ is equal to $W_{F,s/t}(-c b^{-t/s})$ via the reparameterization $y=b^{t/s} x^t$, where division signifies inversion modulo $q-1$.
These sums and their close relatives arise often in number theory \cite{Carlitz-1978, Carlitz-1979, Cochrane-Pinner-2003, Cochrane-Pinner-2011,  Coulter, Davenport-Heilbronn, Karatsuba, Katz-Livne, Kloosterman, Lachaud-Wolfmann, Vinogradow}.  For example, the Kloosterman sum $\sum_{x \in \grmul{F}} \psi(x^{-1}+a x)$ is $-1+W_{F,\card{F}-2}(-a)$.  Our Weil sums in \eqref{Alice} are also of practical interest, as they determine the performance of protocols in communications theory, remote sensing, cryptography, and coding theory.  See the Appendix of \cite{Katz-2012} for how these sums relate to correlation of sequences and nonlinearity of boolean functions.

We are interested in the $p$-adic valuation of $W_{F,d}(a)$.  We extend the $p$-adic valuation $\valp$ from $\Q$ to $\Q(e^{2\pi i/p})$ so that $\valp(1-e^{2\pi i/p})=1/(p-1)$.  Since $W_{F,d}(a)$ always lies in $\Q(e^{2\pi i/p})$, we see that its valuation must be an integer multiple of $1/(p-1)$.
Bounds on the $p$-adic valuation of $W_{F,d}(a)$ have proved very helpful in determining the values of $W_{F,d}(a)$, as can be seen in \cite{Aubry-Katz-Langevin-2014, Aubry-Katz-Langevin-2015, CakCak-Langevin, Calderbank-McGuire-Poonen-Rubinstein, Canteaut-Charpin-Dobbertin-1999, Canteaut-Charpin-Dobbertin-2000-Binary, Canteaut-Charpin-Dobbertin-2000-Weight, Charpin, Dobbertin-Helleseth-Kumar-Martinsen, Feng, Hollmann-Xiang, Hou, Kasami-1971, Katz-2012, Katz-2015, Katz-Langevin-2015, Katz-Langevin-2016, Langevin-1999, Langevin-Veron, Leander-Langevin, McGuire, McGuire-Calderbank}.
The main tool in determining the $p$-adic valuation of these Weil sums is Stickelberger's Theorem on the valuation of Gauss sums, which allows for an exact determination of
\begin{equation}\label{Barbara}
V_{F,d} = \min_{a \in F} \valp(W_{F,d}(a))
\end{equation}
in terms of a combinatorial formula that is given in Lemma \ref{Stanley} below.

When $d$ is a power of $p$ modulo $q-1$, we see that
\begin{equation}\label{Theresa}
W_{F,d}(a) = \sum_{x \in F} \psi((1-a) x) =\begin{cases} \card{F} & \text{if $a=1$,} \\ 0 & \text{otherwise,} \end{cases}
\end{equation}
and $V_{F,d}=[F\colon\Fp]$.  We say that $d$ is {\it degenerate over $F$} in this case because $W_{F,d}$ degenerates to the Weil sum of a monomial or a constant.
If $d$ is degenerate over $F$, then it is also degenerate over any subfield of $F$.
Since $\sum_{a \in F} W_{F,d}(a)=\card{F}$ (see \cite[Corollary 2.6(i)]{Katz-2015}), the case where $d$ is degenerate gives the highest possible value of $V_{F,d}$.
In examining computer calculations of many values of $W_{F,d}(a)$ (see \cite{Langevin-2007}), we noticed that there is a significant gap between the highest values of $V_{F,d}$ observed for nondegenerate $d$ and the value $V_{F,d}=[F\colon\Fp]$ for degenerate $d$.  This observation led us to conjecture that $V_{F,d} \leq (2/3) [F\colon\Fp]$ for nondegenerate $d$, and the main result of this paper is a proof of our conjecture.  We also obtain stronger bounds in specific cases.  One should note that nondegenerate $d$ do not exist when $F=\F_2$, $\F_3$, or $\F_4$.
\begin{theorem}\label{Cecilia}
Let $W_{F,d}(a)$ and $V_{F,d}$ be as defined in \eqref{Alice} and \eqref{Barbara}, where $\gcd(d,q-1)=1$.
\begin{enumerate}[(i)]
\item\label{Timothy} If $d$ is degenerate over $F$, then $V_{F,d}=[F\colon\Fp]$.
\item If $d$ is nondegenerate over $F$, but is degenerate over $\Fp$, and 
\begin{enumerate}[(a)]
\item\label{George} if $[F\colon\Fp]$ is a power of $2$, then $V_{F,d} \leq \frac{1}{2} [F\colon\Fp]$, but
\item\label{Henry} otherwise $ V_{F,d} \leq \frac{2}{3} [F\colon\Fp]$.
\end{enumerate}
\item\label{Nelson} If $d$ is nondegenerate over $\Fp$ (so $p \geq 5$), and
\begin{enumerate}[(a)]
\item\label{David} if $p\equiv 1 \pmod{4}$ and if $[F\colon\Fp]$ is odd, then $V_{F,d} \leq \frac{1}{2} [F\colon\Fp]$, but
\item\label{Edward} otherwise $V_{F,d} \leq \frac{1}{p-1} \ceil{\frac{p-1}{3}} [F\colon\Fp]$.
\end{enumerate}
\end{enumerate}
\end{theorem}
One can see that when $d$ is not degenerate over $F$, the bound $V_{F,d} \leq (2/3) [F\colon\F_p]$ is always true: this universal bound is proved in Theorem \ref{Dorothy}.
The stronger bound in part \eqref{George} when $[F\colon\Fp]$ is a power of $2$ is proved in Theorem \ref{Deidre}.
Interestingly, these two proofs do not use Stickelberger's Theorem, which is the most commonly used tool in determining the $p$-divisibility of these sums.
We do use Stickelberger's Theorem to establish the bounds in part \eqref{Nelson} (proved in Theorem \ref{Natasha}).  Part \eqref{Timothy} follows from \eqref{Theresa}.
\begin{remark}\label{Michael}
The compositum of parts \eqref{George} and \eqref{Nelson} of Theorem \ref{Cecilia} show that if $[F\colon\Fp]$ is a power of $2$ and $d$ is nondegenerate over $F$, then $V_{F,d} \leq (1/2) [F\colon\Fp]$.
\end{remark}
\begin{remark}\label{Francis}
For each case in Theorem \ref{Cecilia} where we have an upper bound for $V_{F,d}$, we now mention those $F$ covered by that case where we know that there exists an exponent $d$ such that $d$ meets the conditions of that case and $V_{F,d}$ equals the upper bound.
\begin{itemize}
\item For every field in case \eqref{George} (see Lemma \ref{Nora} in conjunction with Theorem \ref{Deidre}).
\item In case \eqref{Henry}, if $3 \mid [F\colon\Fp]$ (see Lemma \ref{Gerald}).

\item For every field in case \eqref{David} (see Lemma \ref{Nancy}).
\item In case \eqref{Edward}, if 
\begin{itemize}
\item $p\equiv 1\pmod{3}$ and $3 \nmid [F\colon\Fp]$ (see Lemmata \ref{Peter} and \ref{Paul}); or
\item $p\equiv 2\pmod{3}$ and $2 \nmid [F\colon\Fp]$ (see Lemma \ref{Raphael}).
\end{itemize}
\end{itemize}
\end{remark}
This paper is organized as follows.  After reviewing some basic results in Section \ref{Priscilla}, we prove the universal bound $V_{F,d} \leq (2/3) [F\colon\Fp]$ (when $d$ is nondegenerate over $F$) in Section \ref{Victor}, and then show that this bound is attained whenever $[F\colon\Fp]$ is divisible by $3$.  In Section \ref{Thomas} we prove the bound in part \eqref{George} of Theorem \ref{Cecilia}, where $F$ is obtained from its prime subfield via a tower of quadratic extensions.  In Section \ref{Thomas} we also prove that this bound is always attained for some $d$ in every field $F$ satisfying the hypotheses.  The bounds in part \eqref{Nelson} of Theorem \ref{Cecilia}, when $d$ is known to be nondegenerate over the prime subfield, are proved in Section \ref{Matilda}.  In Section \ref{Roland}, we discuss some open problems.

\section{Preliminaries}\label{Priscilla}

Here we recall some well known results that will be useful in the rest of the paper.
We continue to use the definition of the Weil sum $W_{F,d}(a)$ from \eqref{Alice} and the definition $V_{F,d}$ from \eqref{Barbara} in this section and in the rest of the paper.
\begin{remark}\label{Olaf}
If $d$ is an integer coprime to $q-1$, then $(-1)^d=-1$ in $F$.  This is because the coprimality makes $d$ odd when $p$ is odd.
\end{remark}
\begin{remark}\label{Penelope}
If $d$ and $d'$ are positive integers coprime to $q-1$ such that $d' \equiv d p^k \pmod{q-1}$ for some $k \in \Z$, then $W_{F,d'}(a)=W_{F,d}(a)$ for all $a \in F$.  This is because for every $x \in F$, we have $x^q=x$ and also the absolute trace has $\Tr(x^p)=\Tr(x)$, so that the canonical additive character has $\psi(x^p)=\psi(x)$.
\end{remark}
\begin{remark}\label{Ignatius}
If $d$ and $e$ are positive integers coprime to $q-1$ with $d e \equiv 1 \pmod{q-1}$, then $W_{F,e}(a)=W_{F,d}(a^{-e})$ for every $a \in \grmul{F}$ via the reparameterization mentioned in the Introduction (and use Remark \ref{Olaf} to get the correct sign).  Since $W_{F,d}(0)=W_{F,e}(0)=0$, and $a\mapsto a^{-e}$ is a permutation of $\grmul{F}$, this means that $V_{F,e}=V_{F,d}$.
\end{remark}
\begin{remark}\label{Joshua}
If $d$ is a positive integer with $\gcd(d,q-1)=1$, then it is easy to calculate that $\sum_{a \in F} W_{F,d}(a)=q$ (the first power moment: see \cite[Corollary 2.6(i)]{Katz-2015}), so $V_{F,d} \leq \valp(q) = [F\colon\Fp]$.  And the inequality becomes an equality if $d$ is degenerate over $F$ by \eqref{Theresa}.
\end{remark}
We let $\mchars{F}$ denote the group of multiplicative characters of $F$, and we denote the trivial multiplicative character by $1$ and use the shorthand $\bar{\chi}=\chi^{-1}$.  For $\chi \in \mchars{F}$, we define the Gauss sum
\[
\tau(\chi)=\sum_{a \in \grmul{F}} \psi(a)\chi(a),
\]
where $\psi$ is the canonical additive character of $F$ as defined in the Introduction.
We extend the $p$-adic valuation $\valp$ from $\Q$ to $\Q(e^{2\pi i/p},e^{2\pi i/(q-1)})$ so that $\valp(1-e^{2\pi i/p})=1/(p-1)$.  This enables us to consider the $p$-adic valuation of our Gauss sums.
\begin{lemma}\label{Gary}
Let $d$ be a positive integer with $\gcd(d,q-1)=1$.  Then for $a \in \grmul{F}$, we have
\[
W_{F,d}(a) = \frac{q}{q-1} + \frac{1}{q-1} \sums{\chi\in\mchars{F} \\ \chi\not=1} \tau(\chi) \tau(\bar{\chi}^d) \chi^d(a),
\]
and for $\chi\in\mchars{F}$, we have
\[
\sum_{a\in\grmul{F}} W_{F,d}(a) \bar{\chi}^d(a) = \begin{cases}
q & \text{if $\chi=1$,} \\
\tau(\chi)\tau(\bar{\chi}^d) & \text{otherwise.}
\end{cases}
\]
\end{lemma}
\begin{proof}
The first formula is proved in \cite[eq.~(3)]{Aubry-Katz-Langevin-2015}, and the second is easily obtained from eq.~(4) of the same paper.
\end{proof}
\begin{corollary}\label{Mary}
Let $d$ be a positive integer with $\gcd(d,q-1)=1$.  If $q=2$, then $d$ is degenerate over $F$ and $V_{F,d}=1$.  If $q > 2$, then
\[
V_{F,d} = \mins{\chi\in\mchars{F} \\ \chi\not=1} \valp(\tau(\chi) \tau(\bar{\chi}^d)).
\]
\end{corollary}
\begin{proof}
The $q=2$ case is immediate from Remark \ref{Joshua}, so assume $q > 2$ henceforth.  Note that multiplicative characters take nonzero elements of $F$ to roots of unity in $\C$, which have $p$-adic valuation $0$.
Thus the first formula in Lemma \ref{Gary} shows that
\[
\min_{a \in \grmul{F}} \valp(W_{F,d}(a)) \geq \min\left(\{\valp(q)\}\cup \{\valp(\tau(\chi)\tau(\bar{\chi}^d)) : \chi\in\mchars{F}, \chi\not=1\}\right),
\]
and the reverse inequality follows from the second formula in Lemma \ref{Gary}.
(In essence, the minimum $p$-adic valuation of the Fourier coefficients is the same as the minimum $p$-adic valuation of the original function when $p$ does not divide the order of the underlying group.)
Since $W_{F,d}(0)=0$, we could extend the minimization on the left hand side to include $a=0$.  Thus
\[
V_{F,d} =  \min\left(\{\valp(q)\}\cup \{\valp(\tau(\chi)\tau(\bar{\chi}^d)) : \chi\in\mchars{F}, \chi\not=1\}\right).
\]
So it remains to show that there is some nontrivial $\chi\in\mchars{F}$ such that $\valp(\tau(\chi)\tau(\bar{\chi}^d)) \leq \valp(q)$.  Since $|\tau(\chi)|^2=q$ for any nontrivial multiplicative character and $\tau(\bar{\chi})=\chi(-1) \overline{\tau(\chi)}$ (see \cite[Theorems 5.11, 5.12(iii)]{Lidl-Niederreiter}), and since $d$ is coprime to $q-1$, we see that $\prod_{\chi \not=1} \tau(\chi) \tau(\bar{\chi}^d) \in \{\pm q^{q-2}\}$.  So there is some nontrivial $\chi\in\mchars{F}$ with $\valp(\tau(\chi)\tau(\bar{\chi}^d)) \leq \valp(q)$.
\end{proof}
We also state some useful bounds relating $V_{F,d}$ and $V_{K,d}$ when $K$ is a subfield of $F$.  
\begin{lemma}\label{Lisa}
Let $K$ be a subfield of $F$, and let $d$ be a positive integer with $\gcd(d,\card{\grmul{F}})=1$.
Then $V_{K,d} \leq V_{F,d} \leq [F\colon K] \cdot V_{K,d}$.
\end{lemma}
\begin{proof}
First we prove the lower bound on $V_{F,d}$.
Let $\psi_K$ and $\psi_F$ be the canonical additive characters for $K$ and $F$, respectively, and note that $\psi_F=\psi_K\circ \Tr_{F/K}$, where $\Tr_{F/K}$ is the Galois-theoretic relative trace from $F$ to $K$.
Let $R$ be a set of representatives for the cosets of $\grmul{K}$ in $\grmul{F}$.  Then for any $a \in K$, we have
\begin{align*}
W_{F,d}(a)
& = 1 + \sum_{x \in\grmul{F}} \psi_F(x^d -a x)  \\
& = 1 + \sum_{r \in R} \sum_{y \in \grmul{K}} \psi_K(\Tr_{F/K}((r y)^d - a r y)) \\
& = 1-\card{R} + \sum_{r \in R} \sum_{y \in K} \psi_K(\Tr_{F/K}(r^d) y^d - \Tr_{F/K}(a r) y).
\end{align*}
Now let us consider the values taken on by the inner sum over $K$ in the last expression: these depend on whether the coefficients $\Tr_{F/K}(r^d)$ and $\Tr_{F/K}(a r)$ are zero or not.  If both coefficients are zero, the sum is $\card{K}$, and if only one is zero, then the sum is $0$: since $\gcd(d,\grmul{F})=1$, the map $y\mapsto y^d$ is a permutation of $F$, and thus restricts to a permutation of $K$.  If both coefficients are nonzero, then the reparameterization described in the Introduction shows that the inner sum over $K$ is $W_{K,d}(b)$ for some $b \in K$.  Since $V_{K,d} \leq [K\colon\F_p]=\valp(\card{K})$ by Remark \ref{Joshua}, we see that our inner sum always has a $p$-adic valuation of at least $V_{K,d}$.  Also note that $\card{R}=(\card{F}-1)/(\card{K}-1)$, so that $\card{R}-1$ is a multiple of $\card{K}$, and so it also has a $p$-adic valuation greater than or equal to $V_{K,d}$.  So for any $a \in F$, we see that $W_{F,d}(a)$ has a $p$-adic valuation greater than or equal to $V_{K,d}$, so $V_{F,d} \geq V_{K,d}$.

Now we prove the upper bound on $V_{F,d}$.  If $d$ is degenerate over $K$, then Remark \ref{Joshua} shows that $[F\colon K] \cdot V_{K,d}=[F\colon\F_p]$, and so our upper bound is true by another application of Remark \ref{Joshua}.  So from now on we assume that $d$ is nondegenerate over $K$ (which forces $\card{K}>2$).  Let $N_{F/K}$ denote the Galois-theoretic relative norm from $K$ to $F$.  Then the Davenport-Hasse relation \cite[Theorem 5.14]{Lidl-Niederreiter} tells us that for $\chi\in\mchars{K}$, we have
\[
-\tau(\chi\circ N_{F/K}) = (-\tau(\chi))^{[F\colon K]}.
\]
As $\chi$ runs through the nontrivial characters in $\mchars{K}$, their lifts $\chi\circ N_{F/K}$ run through the nontrivial characters in $\mchars{F}$ whose orders are divisors of $\card{\grmul{K}}$.
So Corollary \ref{Mary} and the Davenport-Hasse relation show us that
\begin{align*}
V_{F,d}
& \leq \mins{\chi \in \mchars{K} \\ \chi\not=1} \valp(\tau(\chi\circ N_{F/K}) \tau((\overline{\chi\circ N_{F/K}})^d)) \\
& = \mins{\chi \in \mchars{K} \\ \chi\not=1} [F\colon K] \valp(\tau(\chi) \tau(\bar{\chi}^d)) \\
& = [F\colon K] \cdot V_{K,d}. \qedhere
\end{align*}
\end{proof}
\begin{remark} The special case of Lemma \ref{Lisa} for characteristic $2$ was proved in \cite[Theorem 7.1]{Canteaut-Charpin-Dobbertin-2000-Weight} using McEliece's Theorem on $2$-divisibility of weights of words in cyclic codes {\rm (}see \cite[Corollary to Theorem 2]{McEliece}{\rm )}, which is closely related to Stickelberger's Theorem.
Note that our proof in this paper of the general case uses neither Stickelberger's nor McEliece's Theorem.
\end{remark}

Now we provide the definitions needed to make use of Stickelberger's Theorem.
If $t$ is an integer with $t \geq 2$ and $n$ is a positive integer, and if $a \in \Z/(t^n-1)\Z$, we define the {\it standard $t$-ary expansion of $a$} to be the expression
\[
a=a_0 t^0 + a_1 t^1 + \cdots + a_{n-1} t^{n-1},
\]
where the powers of $t$ are elements of $\Z/(t^n-1)\Z$ and $a_0,\ldots,a_{n-1}$ are elements of $\Z$ with $0 \leq a_i < t$ for every $i$, and where we insist that $a_0=\cdots=a_n=0$ when $a=0$ (to make the $a_i$'s uniquely defined).
Then we define the {\it $t$-ary weight of $a \in \Z/(t^n-1)\Z$},
\begin{equation}\label{Walter}
\wt_{t,n}(a) = a_0+a_1+\cdots+a_{n-1},
\end{equation}
so that $\wt_{t,n} \colon \Z/(t^n-1)\Z \to \Z$ with $0 \leq \wt_{t,n}(a) < n(t-1)$ for every $a \in \Z/(t^n-1)\Z$.
Note that our weight function is subadditive, that is, $\wt_{t,n}(a+b) \leq \wt_{t,n}(a)+\wt_{t,n}(b)$ for all $a, b \in \Z/(t^n-1)\Z$.  And this inequality becomes an equality if and only if there are no carries (and no cyclic carries from the $(n-1)$th digit to the $0$th digit) when we compute the sum of $a$ and $b$ in base $t$ arithmetic by summing their standard $t$-ary expansions.
\begin{lemma}\label{Stanley}
Let $q=p^n > 2$, and let $d$ be a positive integer with $\gcd(d,q-1)=1$.  Let
\[
m = \mins{a \in \Z/(p^n-1)\Z \\ a \not=0} \wt_{p,n}(a)+\wt_{p,n}(-d a),
\]
or equivalently,
\[
m = n(p-1) + \mins{a \in \Z/(p^n-1)\Z \\ a \not=0} \wt_{p,n}(d a)-\wt_{p,n}(a).
\]
Then $V_{F,d} = m/(p-1)$.
\end{lemma}
\begin{proof}
One can see that our two definitions of $m$ are equivalent by reparameterizing the first one with $-a$ in place of $a$ and then noting that for any nonzero $a \in \Z/(p^n-1)\Z$, we have $\wt_{p,n}(-a)=n(p-1)-\wt_{p,n}(a)$.  Now let $\zeta=e^{2\pi i/(q-1)}$, and identify $F$ with $\Z[\zeta]/P$, where $P$ is a prime ideal of $\Z[\zeta]$ containing $p$.  Then $\zeta+P$ is a primitive element of $F$.  Let $\omega \colon \grmul{F} \to \Q(\zeta)$ be the Teichm\"uller character, which is determined by $\omega(\zeta+P)=\zeta$.  Then Stickelberger's Theorem \cite[Theorem 2.1]{Lang} says that $\val(\tau(\omega^{-a}))=\wt_{p,n}(a)/(p-1)$ for every $a \in \Z/(p^n-1)\Z$.  And $\omega^{-a}$ runs through the nontrivial multiplicative characters of $F$ as $a$ runs through the nonzero elements of $\Z/(p^n-1)\Z$, so the expression for $V_{F,d}$ in Corollary \ref{Mary} becomes the desired expression.
\end{proof}
\begin{corollary}\label{Ursula}
Let $d$ be a positive integer with $\gcd(d,q-1)=1$.  Then $V_{F,d}=[F\colon\F_p]$ if and only if $d$ is degenerate over $F$.  If $d$ is nondegenerate over $F$, then $2/(p-1) \leq V_{F,d} < [F\colon\F_p]$ and $V_{F,d}=2/(p-1)$ if and only if $-d$ is congruent to a power of $p$ modulo $q-1$.  
\end{corollary}
\begin{proof}
Remark \ref{Joshua} handles the degenerate case.  So we suppose $d$ is nondegenerate over $F$, which forces $q=p^n >4$.

For the upper bound, look at the formula for $m$ in Lemma \ref{Stanley}.  Since $d$ is not a power of $p$ modulo $\Z/(p^n-1)\Z$, we see that $\wt_{p,n}(-d) < n(p-1)-1$, and then $m \leq \wt_{p,n}(1)+\wt_{p,n}(-d) < n(p-1)$.  Thus by Lemma \ref{Stanley}, we know that $V_{F,d} < n=[F\colon\F_p]$.

In the formula for $m$ in Lemma \ref{Stanley}, note that $a$ is nonzero and $d$ is coprime to $p^n-1$, so both $\wt_{p,n}(a)$ and $\wt_{p,n}(-d a)$ are always strictly positive integers.  And these weights are both equal to $1$ simultaneously if and only if both $a$ and $-d a$ are powers of $p$ modulo $q-1$, which will occur for some $a \in \Z/(q-1)\Z$ if and only if $-d$ is a power of $p$ modulo $q-1$.  This proves our lower bound and shows us when it is achieved.
\end{proof}
\begin{corollary}
Let $d$ be a positive integer with $\gcd(d,q-1)=1$.  Then $V_{F,d}+V_{F,-d} \leq [F\colon\F_p]+\frac{2}{p-1}$.
\end{corollary}
\begin{proof}
If $q=2$, then both $d$ and $-d$ are degenerate over $F$, and then our inequality follows from Corollary \ref{Ursula}, so assume $q>2$ henceforth.

By Lemma \ref{Stanley}, we know that $(p-1)V_{F,d} \leq 1+\wt_{p,n}(-d)$ and $(p-1) V_{F,-d} \leq 1+\wt_{p,n}(d)$.  So $(p-1) (V_{F,d}+V_{F,-d}) \leq 2 +\wt_{p,n}(d) + \wt_{p,n}(-d)$.  Since $d$ is coprime to $q-1$ and $q>2$, we see that $d$ must be nonzero modulo $q-1$, and so $\wt_{p,n}(d)+\wt_{p,n}(-d)=n(p-1)$.  Thus $V_{F,d}+V_{F,-d} \leq n + 2/(p-1)$.
\end{proof}

\section{Proof of the Universal Upper Bound}\label{Victor}

In this section, we prove an upper bound on $V_{F,d}$ that holds whenever $d$ is nondegenerate over $F$, and then show that our bound is attained for infinitely many fields $F$.

\begin{theorem}\label{Dorothy}
If $F$ is a finite field of characteristic $p$ and order $q$, and $d$ is a positive integer with $\gcd(d,q-1)=1$ such that $d$ is not degenerate over $F$, then $V_{F,d} \leq \frac{2}{3} [F\colon\Fp]$.
\end{theorem}
\begin{proof}
Let $\psi\colon F \to \C$ be the canonical additive character of $F$, and for $u, a \in F$, define
\[
\mywalsh u a = \sum_{x\in F} \psi(u x^d - a x).
\]
Then for $u,v,w,a \in F$, we have
\begin{align*}
\frac{1}{q} \sum_{a\in F} \mywalsh u a \mywalsh v a \mywalsh w a
&= \frac{1}{q} \sum_{r,s,t,a \in F} \psi( u r^d + v s^d + w t^d)\psi(-a(r+s+t))\\
&= \sums{r,s,t \in F \\ r+s+t=0} \psi( u r^d + v s^d + w t^d) \\
&= \sums{r,s \in F \\ r+s=0} \psi(u r^d + v s^d) + \sums{r,s \in F \\ t\in\grmul{F} \\ r+s=-t} \psi(u r^d + v s^d + w t^d)\\
& = \sum_{r \in F} \psi((u-v) r^d)+ \sums{x,y \in F \\ t \in \grmul{F} \\ x+y=-1} \psi((u x^d + v y^d + w) t^d) \\
& = q \delta_{u,v}- q + \sums{x,y,t \in F \\ x+y=-1} \psi((u x^d + v y^d + w) t^d),
\end{align*}
where $\delta$ is the Kronecker delta, and where the first sum in the fourth step is obtained using Remark \ref{Olaf}, while the second sum in the fourth step is obtained via the reparameterization $(r,s)=(t x,t y)$.

For $u,v,w \in F$, we define $N(u,v,w)$ to be the number of $(x,y) \in F^2$ simultaneously satisfying $x+y=-1$ and $u x^d + v y^d=-w$, and we define $f_{v,u}(x)=v(x+1)^d-u x^d$.  Then
\begin{align*}
N(u,v,w) & = \card{\left\{(x,y) \in F^2: x+y=-1, u x^d+v y^d=-w\right\}} \\
& = \card{\left\{x \in F: f_{v,u}(x)=w\right\}}.
\end{align*}
Since our assumption that $\gcd(d,q-1)=1$ makes $t\mapsto t^d$ a permutation of $F$, our calculation in the previous paragraph shows that
\[
\sum_{a\in F} \mywalsh u a \mywalsh v a \mywalsh w a =q^2 (\delta_{u,v} - 1 + N(u,v,w)).
\]
If we could find $u,v,w \in F$ with $u\not=v$ and $p\mid N(u,v,w)$, then we would have $p\nmid(\delta_{u,v}-1+N(u,v,w))$, and so
\[
\valp\left(\sum_{a \in F}  \mywalsh u a \mywalsh v a \mywalsh w a\right) = 2 [F\colon\F_p],
\]
which would imply that $\valp(\mywalsh b a) \leq 2[F\colon\Fp]/3$ for some $(b,a) \in F^2$, and in fact, this must be true for some $b \in \grmul{F}$ because direct calculation from the definition of $\mywalsh u a$ shows that
\[
\mywalsh 0 a = \begin{cases}
q & \text{if $a=0$,} \\
0 & \text{otherwise,}
\end{cases}
\]
so that $\valp(\mywalsh 0 a) \geq [F\colon \Fp]$ for all $a \in F$.  And when $b\not=0$, then $\mywalsh b a = W_{F,d}(b^{-1/d} a)$ (where $1/d$ signifies the multiplicative inverse of $d$ modulo $q-1$), so this would prove our claim.

Since $N(u,v,w)$ counts the zeroes of a polynomial, it is always a nonnegative integer.  For any $(u,v) \in F^2$, we observe that
\[
\sum_{w \in F} N(u,v,w) = q.
\]
So the only way that we can have $p\nmid N(u,v,w)$ for every $w \in F$ is if $N(u,v,w)=1$ for every $w \in F$, and this is true if and only if $x\mapsto f_{v,u}(x)$ is a permutation of $F$.
So we will have completed our proof if we can find some $u,v \in F$ with $u\not=v$ such that $x\mapsto f_{v,u}(x)$ is not a permutation of $F$.
In fact, we shall show that for every $v \in \grmul{F}$, there is some $u \in F\smallsetminus\{v\}$ such that $x\mapsto f_{v,u}(x)$ is not a permutation of $F$.

Let $v \in \grmul{F}$, and suppose that $x\mapsto f_{v,u}(x)$ is a permutation of $F$ for every $u \in F \smallsetminus\{v\}$ in order to show a contradiction. 
Thus for any $u,x,y \in F$ with $u\not=v$ and $x\not=y$, we have $f_{v,u}(x)\not=f_{v,u}(y)$, and thus
\[
\frac{f_{v,0}(y)-f_{v,0}(x)}{y^d-x^d} \not = u,
\]
where we have used the fact that $x^d\not=y^d$ because $z\mapsto z^d$ is a permutation of $F$.
So we must have
\[
\frac{f_{v,0}(y)-f_{v,0}(x)}{y^d-x^d}  = v,
\]
for every $x,y \in F$ with $x\not=y$.
This means that $x\mapsto f_{v,v}(x)$ is a constant function on $F$, and in fact, the constant must be $v$ since $f_{v,v}(0)=v$.  Since $v\not=0$, this means that the function $x\mapsto f_{1,1}(x)-1=(x+1)^d-x^d-1$ must be the zero function on $F$.
Let $d'$ be the least positive integer congruent to $d$ modulo $q-1$, and write $d'=p^k e$ for some nonnegative integer $k$ and positive integer $e$ with $p\nmid e$.  So $d \equiv p^k e \pmod{q-1}$ with $0 < e < q$, and in fact $e > 1$ because $d$ is nondegenerate over $F$.
For any $x \in F$, we have $x^d=x^{p^k e}$, so that $0=f_{1,1}(x)-1=(x+1)^{p^k e}-x^{p^k e}-1=((x+1)^e-x^e-1)^{p^k}$, so that $x\mapsto (x+1)^e-x^e-1$ is the zero function on $F$.
But
\[
(x+1)^e-x^e-1 = e x^{e-1} + \sum_{j=1}^{e-2} \binom{e}{j} x^j,
\]
and since $p\nmid e$, this is a polynomial of degree $e-1$ with $0 < e-1 < q$, and so is not divisible by $x^q-x$, so that $x\mapsto (x+1)^e-x^e-1$ is not the zero function on $F$.  This contradiction shows that $f_{v,u}(x)$ is not a permutation of $F$ for some $u\in F\smallsetminus\{v\}$, which completes our proof.
\end{proof}
The following result shows that infinitely many fields $F$ have some exponent $d$ such that the upper bound in Theorem \ref{Dorothy} is attained.
\begin{lemma}\label{Gerald}
Suppose that $[F\colon\Fp]$ is not a power of $2$.  Let $\ell$ be the smallest odd prime divisor of $[F\colon\Fp]$.  Then there is some $d$ with $\gcd(d,q-1)=1$ such that $d$ is nondegenerate over $F$ but degenerate over $\Fp$ and $V_{F,d} = \frac{\ell+1}{2 \ell} [F\colon\Fp]$.
\end{lemma}
\begin{proof}
Recall that $n=[F:\F_p]$.  If $p=2$, and $d=2^{n/\ell}+1$ or $d=2^{2 n/\ell}-2^{n/\ell}+1$, then it is known (see \cite{Gold}, \cite[Theorem 5]{Kasami-1966}, \cite[Remark 3]{Kasami-1971}, \cite[Theorem 16]{Kasami-Lin-Peterson}, \cite{Welch}) that $\{W_{F,d}(a): a \in F\}=\{0,\pm \sqrt{2^{n(\ell+1)/\ell}}\}$, so $V_{F,d}=n(\ell+1)/(2 \ell)$.  If $p$ is odd and $d=(p^{2 n/\ell}+1)/2$ or $d=p^{2 n/\ell}-p^{n/\ell}+1$, then it is known (see \cite{Helleseth-1971}, \cite[Theorem 4.9]{Helleseth-1976}, \cite[Theorems 3.3, 3.4]{Trachtenberg}) that $\{W_{F,d}(a): a \in F\}=\{0,\pm \sqrt{p^{n(\ell+1)/\ell}}\}$, so $V_{F,d}=n(\ell+1)/(2 \ell)$.

Now we demonstrate that each of these exponents is coprime to the order of the group of units for its field.  First of all, $\gcd(2^{n/\ell}+1,2^n-1)=\gcd(2^{n/\ell}+1,(-1)^\ell-1)=\gcd(2^{n/\ell}+1,-2)=1$.  And if $p$ is odd, then $(p^{2 n/\ell}+1)/2$ is odd, so that
\begin{align*}
\gcd\left(\frac{p^{2 n/\ell}+1}{2},p^n-1\right)
& = \frac{1}{2} \gcd(p^{2 n/\ell}+1,p^n-1) \\
& = \frac{1}{2} \gcd(p^{2 n/\ell}+1,(-1)^{(\ell-1)/2} p^{n/\ell} -1) \\
& = \frac{1}{2} \gcd(1+1,(-1)^{(\ell-1)/2} p^{n/\ell}-1) \\
& = 1.
\end{align*}
Now we examine $d=p^{2 n/\ell}-p^{n/\ell}+1$ for an arbitrary prime $p$.
We write the least odd prime divisor $\ell$ of $n$ as $\ell=3 k + r$ with $k \in \Z$ and $r \in \{0,1,2\}$.  Then $(p^{2 n/\ell}-p^{n/\ell}+1)(p^{n/\ell}+1)=p^{3 n/\ell}+1$, so then $\gcd(p^{2 n/\ell}-p^{n/\ell}+1,p^n-1) = \gcd(p^{2 n/\ell}-p^{n/\ell}+1,(-1)^k p^{r n/\ell}-1)$.  If $r=0$, then $\ell=3$, so then $k=1$, and then our greatest common divisor is $\gcd(p^{2 n/3}-p^{n/3}+1,-2)=1$.  If $r=1$, then $k$ must be even (no prime is $4$ modulo $6$), so then our greatest common divisor becomes $\gcd(p^{2 n/\ell}-p^{n/\ell}+1,p^{n/\ell}-1)=\gcd(1-1+1,p^{n/\ell}-1)=1$.  And if $r=2$, then $k$ must be odd (no odd prime is $2$ modulo $6$), so our greatest common divisor becomes $\gcd(p^{2 n/\ell}-p^{n/\ell}+1,-p^{2 n/\ell}-1)=\gcd(p^{n/\ell},-p^{2 n/\ell}-1)=\gcd(p^{n/\ell},-1)=1$.

Since $V_{F,d} < [F\colon\Fp]$ for each of the four $d$ proposed above, it is clear that these values of $d$ are all nondegenerate over $F$.  And our exponents are clearly degenerate over $\Fp$: all exponents are degenerate over $\F_2$, and when $p$ is odd, both $(p^{2 n/\ell}+1)/2=1+(p^{n/\ell}-1)(p^{n/\ell}+1)/2$ and $p^{2 n/\ell}-p^{n/\ell}+1=1+(p^{n/\ell}-1)p^{n/\ell}$ are congruent to $1$ modulo $p-1$.
\end{proof}

\section{Towers of Quadratic Extensions}\label{Thomas}

In this section, we prove an upper bound for $V_{F,d}$ when the degree of $F$ over its prime subfield $\Fp$ is a power of $2$, and when $d$ is nondegenerate over $F$, but is degenerate over $\Fp$.  (This is part \eqref{George} of Theorem \ref{Cecilia}.)  Then we show that for any such $F$, there is some such $d$ for which our upper bound is attained.
\begin{theorem}\label{Deidre}
Let $[F\colon\Fp]=2^s$ for some $s \geq 1$.  Let $d$ be a positive integer with $\gcd(d,q-1)=1$ such that $d$ is not degenerate over $F$, but is degenerate over $\Fp$.  Then $V_{F,d} \leq \frac{1}{2} [F\colon\Fp]$, and this is always an equality if $d$ is degenerate in the subfield of $F$ of order $\sqrt{q}$.
\end{theorem}
\begin{proof}
Let $L$ be the smallest subfield of $F$ (possibly equal to $F$) such that $d$ is not degenerate over $L$, and let $K$ be the unique subfield of $L$ with $[L\colon K]=2$.  Then by \cite[Corollary 4.4]{Aubry-Katz-Langevin-2015}, we have $V_{L,d}=[K\colon\Fp]$.
Lemma \ref{Lisa} then shows that $V_{F,d} \leq [F\colon L] [K\colon\Fp] = \frac{1}{2} [F\colon\Fp]$.
If $d$ is degenerate in the subfield of $F$ of order $\sqrt{q}$, then $L=F$, and so $V_{F,d}=V_{L,d}=[K:\F_p]=\frac{1}{2} [L:\F_p]=\frac{1}{2} [F:\F_p]$.
\end{proof}
The following result shows that if $F$ is as described in Theorem \ref{Deidre}, then there is always an exponent $d$ as described in the same theorem such that the upper bound on $V_{F,d}$ in the theorem is attained.  Note that since there are no nondegenerate exponents over $\F_2$, $\F_3$, and $\F_4$, we need to work with a field having more than four elements to satisfy the hypotheses of Theorem \ref{Deidre}.
\begin{lemma}\label{Nora}
Let $q > 4$ and $[F\colon\Fp]=2^s$ for some $s \geq 1$.  Let $K$ be the unique subfield of $F$ of order $\sqrt{q}$.  Then there is always some positive integer $d$ with $\gcd(d,q-1)=1$ such that $d$ is nondegenerate over $F$ but degenerate over $K$.  For examples of such $d$, we have
\begin{itemize}
\item $d=q-\sqrt{q}-1$ when $\sqrt{q} \pmod{3} \in\{0,1\}$,
\item $d=(q+2)/3$ when $\sqrt{q} \pmod{9} \in \{2,8\}$, and 
\item $d=(2 q+1)/3$ when $\sqrt{q} \pmod{9} \in \{5,8\}$.
\end{itemize}
\end{lemma}
\begin{proof}
We prove this result by proving that the above examples satisfy the desired conditions.

First suppose $\sqrt{q} \pmod{3} \in \{0,1\}$ and $d=q-\sqrt{q}+1$.
Note that $d\equiv 1\pmod{\sqrt{q}-1}$, so $d$ is degenerate over $K$ and $\gcd(d,\sqrt{q}-1)=1$.
Also note that $\gcd(d,\sqrt{q}+1)=\gcd(1+1+1,\sqrt{q}+1)=1$ since $3\nmid \sqrt{q}+1$.  Thus $\gcd(d,q-1)=1$.
Finally, note that $q/p < q-\sqrt{q}+1 < q-1$ because $q > 4$ so that it is clear that $d=q-\sqrt{q}+1$ is not a power of $p$ modulo $q-1$.  So $d$ is nondegenerate over $F$.

Now suppose that $\sqrt{q} \pmod{9} \in \{2,8\}$ and $d=(q+2)/3$.
Note that $3 d \equiv 3 \pmod{\sqrt{q}-1}$ and since $3\nmid \sqrt{q}-1$, this means that $d\equiv 1 \pmod{\sqrt{q}-1}$, so $d$ is degenerate over $K$.
Also note that $\gcd(3 d,q-1)=\gcd(3,q-1)$ but $3\nmid d$, so $\gcd(d,q-1)=1$.
If $p \geq 3$, then $q/p < (q+2)/3 < q-1$ because $q > 4$, so it is clear that $d=(q+2)/3$ is not a power of $p$ modulo $q-1$.
And if $p=2$, then $q/4 < (q+2)/3 < q/2 < q-1$ because $q > 4$, so it is clear that $d=(q+2)/3$ is not a power of $p$ modulo $q-1$.  So $d$ is nondegenerate over $F$.

Now suppose that $\sqrt{q} \pmod{9} \in \{5,8\}$ and $d=(2 q+1)/3$.
Note that $3 d \equiv 3 \pmod{\sqrt{q}-1}$ and since $3\nmid \sqrt{q}-1$, this means that $d\equiv 1 \pmod{\sqrt{q}-1}$, so $d$ is degenerate over $K$.
Also note that $\gcd(3 d,q-1)=\gcd(3,q-1)$ but $3\nmid d$, so $\gcd(d,q-1)=1$.
And $q/p < (2 q+1)/3 < q-1$ because $q > 4$, so it is clear that $d=(2 q+1)/3$ is not a power of $p$ modulo $q-1$.  So $d$ is nondegenerate over $F$.
\end{proof}

\section{Exponents that are Nondegenerate over the Prime Subfield}\label{Matilda}

In this section, we obtain upper bounds on $V_{F,d}$ when $d$ is nondegenerate over the prime subfield $\Fp$ of $F$.  (This is part \eqref{Nelson} of Theorem \ref{Cecilia}.)  Our results here also show that these upper bounds are attained in infinitely many of the fields under consideration: see Remark \ref{Francis} for details.  We prove our bounds in Theorem \ref{Natasha}, which itself is based on a few other results (Propositions \ref{Angela}--\ref{Samuel} and Lemma \ref{Nancy}), which in turn are based on further technical results (Lemmata \ref{Peter}--\ref{Raphael}).  To show the motivation for the various results, we present Theorem \ref{Natasha} first, then proceed to the intermediate results upon which our proof of Theorem \ref{Natasha} directly depends, and conclude with proofs of the tributary claims that we use to establish the intermediate results.  The lemmata in this section also show that the upper bound in Theorem \ref{Natasha} is met infinitely often (cf.~Remark \ref{Francis}).
\begin{theorem}\label{Natasha}
Let $d$ be a positive integer with $\gcd(d,q-1)=1$ and $d\not\equiv 1\pmod{p-1}$.  If $d \equiv (p+1)/2 \pmod{p-1}$ (which can only happen if $p\equiv 1 \pmod{4}$) and if $n$ is odd, then $V_{F,d} \leq \frac{1}{2} [F\colon\Fp]$.  Otherwise $V_{F,d} \leq \frac{1}{p-1}\ceil{\frac{p-1}{3}} [F\colon\Fp]$.
\end{theorem}
\begin{proof}
Our conditions on $d$ imply that $p \geq 5$.
If $d\not\equiv (p+1)/2 \pmod{p-1}$, then Propositions \ref{Angela} and \ref{Bernard} below show that $V_{\Fp,d} \leq \frac{1}{p-1} \ceil{\frac{p-1}{3}}$, and Lemma \ref{Lisa} shows that $V_{F,d} \leq V_{\Fp,d} \cdot [F\colon\Fp]$, thus securing our bound.

So suppose $d\equiv (p+1)/2 \pmod{p-1}$ henceforth.  We must have $p\equiv 1 \pmod{4}$, since otherwise $d$ would be even, hence not coprime to $q-1$.  By Lemma \ref{Nancy} below, we have $V_{\Fp,d}=1/2$, and Lemma \ref{Lisa} shows that $V_{F,d} \leq V_{\Fp,d} \cdot [F\colon\Fp]$, thus securing the desired bound when $n$ is odd, and also the desired bound when $n$ is even and $p=5$, since $\frac{1}{5-1} \ceil{\frac{5-1}{3}} = \frac{1}{2}$.

So suppose that $n$ is even and $p > 5$ henceforth.  Since $p\equiv 1\pmod{4}$, this means that $p\geq 13$.  So Proposition \ref{Samuel} below shows that $V_{\F_{p^2},d} \leq 1/2$.  Then we use Lemma \ref{Lisa} to see that $V_{F,d} \leq \frac{1}{2} [F\colon\F_{p^2}] = \frac{1}{4} [F\colon\Fp] < \frac{1}{p-1} \ceil{\frac{p-1}{3}} [F\colon\Fp]$.
\end{proof}
Now we prove four results upon which our proof of Theorem \ref{Natasha} depends.
\begin{proposition}\label{Angela}
Let $F$ be a prime field of order $p$ with $p \equiv 1 \pmod{3}$.  If $d$ is a positive integer with $\gcd(d,p-1)=1$ and $d\not\equiv 1 \pmod{(p-1)/2}$, then $V_{F,d} \leq 1/3$.  Equality is achieved if and only if one of the following holds:
\begin{itemize}
\item $p\not\equiv 7 \pmod{9}$ and $d \equiv (p+2)/3 \pmod{p-1}$; or
\item $p\not\equiv 4 \pmod{9}$ and $d \equiv (2 p+1)/3 \pmod{p-1}$; or
\item $p=19$ and $d \equiv 5$ or $11 \pmod{18}$.
\end{itemize}
\end{proposition}
\begin{proof}
The conditions on $p$ and $d$ imply that $p \geq 7$ and $d \geq 5$.
And we may assume that $d < p-1$ by replacing it with its remainder upon division by $p-1$.
We shall use Lemma \ref{Stanley} to bound $V_{F,d}$.
Let $\wt=\wt_{p,1}$ be the $p$-ary weight function on $\Z/(p-1)\Z$ as defined in \eqref{Walter}, and note that for any $a \in \Z/(p-1)\Z$, the value of $\wt(a)$ is equal to the unique integer in $\{0,1,\ldots,p-2\}$ that is congruent to $a$ modulo $p-1$.
We use the convention that if $z \in \Z$, then $\wt(z)$ is a shorthand for $\wt(\bar{z})$, where $\bar{z} \in \Z/(p-1)\Z$ is the reduction modulo $p-1$ of $z$.

Now write $u=p-1$ and note that $u/6$ is an integer since $p\equiv 1 \pmod{6}$ and $u \geq 6$ because $p\geq 7$.  Write $u=d a+ r$ with a positive integral quotient $a$ and remainder $r \in \{1,\ldots,d-1\}$, so then $a=(u-r)/d$.
Then $r\equiv -d a \pmod{u}$, and so
\[
\wt(a)+\wt(-d a) = \wt(a)+\wt(r) = a+r = \frac{u-r}{d}+r.
\]
Lemma \ref{Stanley} will prove our desired bound ($V_{F,d} \leq 1/3$) if we show that
\[
\frac{u-r}{d}+r  \leq \frac{u}{3},
\]
which is equivalent to
\[
(d-1) r \leq (d-3) \frac{u}{3}.
\]
If we have $d < (u/3)-2$, then since $r \leq d-1$ and $d \geq 5$, we have $(d-1) r \leq (d-1)^2 \leq (d-3)(d+3) \leq (d-3)(u/3)$, which means that our inequality is satisfied, and equality is not actually possible (we would need $d=5$ and $u/3=d+3$, which would force $u=24$, that is, $p=25$, which is absurd).

We cannot have $d=(u/3)-2$, since that would make $d$ even, hence not coprime to $u$.

If we have $d=(u/3)-1$, then note that $p \geq 19$ since we have $d \geq 5$.
Furthermore, we have $-3 d \equiv 3 \pmod{u}$, so then
\[
\wt(3)+\wt(-3 d) = 3 + 3 = 6,
\]
and this is less than or equal to $u/3$ because $p\geq 19$ (with equality only when $p=19$ and $d=5$).

We cannot have $d=u/3$, for then $\gcd(d,u)=u/3 > 1$.

If we have $(u/3)+1 \leq d < u/2$, then $a=2$ and $r=u-2 d$, so then $\wt(a)+\wt(r)=2+u-2 d \leq u/3$ with equality if and only if $d=(u/3)+1=(p+2)/3$.

We cannot have $d=u/2$, for then $\gcd(d,u)=u/2 > 1$.

We cannot have $d=(u/2)+1=(p+1)/2$, because we are assuming $d \not\equiv 1\pmod{(p-1)/2}$.

If we have $(u/2)+2 \leq d < 2 u/3$, then write $d=(u/2)+e$, so that $2 \leq e < u/6$.
And write $u-d=2 e b + s$ with $0 \leq s < 2 e$, so that $b=\frac{(u/2)-e-s}{2 e} \leq \frac{u}{8}-\frac{1}{2}$.  Since $2 d \equiv 2 e \pmod{u}$, we have $-d(2 b+1)\equiv -2 b e -d \equiv s \pmod{u}$, and then
\[
\wt(2 b+1) + \wt(-d(2 b+1)) = \wt(2 b+1) + \wt(s) = 2 b + 1 + s
\]
because $2 b+1 \leq \frac{u}{4} = \frac{p-1}{4} < p-1$ and $s < 2 e < \frac{u}{3}=\frac{p-1}{3} < p-1$.
So to prove our upper bound on $V_{\Fp,d}$ it suffices to show that $2 b+1+s \leq u/3$, that is,
\[
\frac{(u/2)-s}{e} + s \leq \frac{u}{3},
\]
or equivalently,
\[
(e-1) s \leq (2 e-3) \frac{u}{6}.
\]
This is in fact always satisfied: since $s \leq 2 e-1$ and $2 \leq e < \frac{u}{6}$, we have $(e-1) s \leq (e-1) (2 e-1) \leq (2 e-3)(e+1) \leq (2 e-3) (u/6)$, and we can only have exact equality if $2=e=(u/6)-1$, that is, if $p=19$ and $d=(u/2)+e=11$.

We cannot have $d=2 u/3$, for then $\gcd(d,u)=u/3 > 1$.

If we have $d \geq (2 u/3)+1$, then $a=1$ and $r=u-d$ and $\wt(a)+\wt(r)=1+u-d \leq u/3$, with exact equality only if $d=(2 u/3)+1=(2 p+1)/3$.

Thus we have proved our upper bound on $V_{\Fp,d}$ for all values of $d$, and have shown that equality can be achieved only if $d=(u/3)+1=(p+2)/3$ or $d=(2 u/3)+1=(2 p+1)/3$ or else if $p=19$ and $d \in \{5,11\}$.  For the former two cases, Lemmata \ref{Peter} and \ref{Paul} tell us precisely when $\gcd(d,p-1)=1$, and show that $V_{\Fp,d}=1/3$ in these cases.
And it is easy to check that $V_{\F_{19},5}=V_{\F_{19},11}=1/3$ by direct computation of the value $m$ from Lemma \ref{Stanley}.
\end{proof}

\begin{proposition}\label{Bernard}
Let $F$ be a prime field of order $p$ with $p$ odd and $p \equiv 2 \pmod{3}$.
 If $d$ is a positive integer with $\gcd(d,p-1)=1$ and $d\not\equiv 1 \pmod{(p-1)/2}$, then $V_{F,d} \leq \frac{p+1}{3(p-1)}$.  Equality is achieved if and only if one of the following holds:
\begin{itemize}
\item $d \equiv 3 \pmod{p-1}$, or
\item $d \equiv (2p-1)/3 \pmod{p-1}$.
\end{itemize}
\end{proposition}
\begin{proof}
The conditions on $p$ and $d$ imply that $p \geq 5$ and $d \geq 3$.
If $d=3$, then Lemma \ref{Raphael} below shows that $V_{\Fp,3}=(p+1)/(3(p-1))$.  So we assume that $d \geq 5$ henceforth.
And we may assume that $d < p-1$ by replacing it with its remainder upon division by $p-1$.
We shall use Lemma \ref{Stanley} to bound $V_{F,d}$.
Let $\wt=\wt_{p,1}$ be the $p$-ary weight function on $\Z/(p-1)\Z$ as defined in \eqref{Walter}, and note that for any $a \in \Z/(p-1)\Z$, the value of $\wt(a)$ is equal to the unique integer in $\{0,1,\ldots,p-2\}$ that is congruent to $a$ modulo $p-1$.
We use the convention that if $z \in \Z$, then $\wt(z)$ is a shorthand for $\wt(\bar{z})$, where $\bar{z} \in \Z/(p-1)\Z$ is the reduction modulo $p-1$ of $z$.

Now write $u=p-1$ and note that $(u+2)/6$ is an integer since $p\equiv 5 \pmod{6}$ and $u \geq 4$ because $p\geq 5$.  Write $u=d a+ r$ with a positive integral quotient $a$ and remainder $r \in \{1,\ldots,d-1\}$, so then $a=(u-r)/d$.
Then $r\equiv -d a \pmod{u}$, and so
\[
\wt(a)+\wt(-d a)=\wt(a)+\wt(r)=a+r=\frac{u-r}{d}+r.
\]
Lemma \ref{Stanley} will prove our desired bound ($V_{F,d} \leq \frac{p+1}{3(p-1)}=\frac{u+2}{3(p-1)}$) if we show that
\[
\frac{u-r}{d}+r  \leq \frac{u+2}{3},
\]
which is equivalent to
\[
(d-1) r -2 \leq (d-3) \frac{u+2}{3}.
\]
If we have $d < (u-1)/3$, then since $r \leq d-1$ and $d \geq 5$, we have $(d-1) r -2\leq (d-1)^2 -2 \leq (d-3)(d+2) \leq (d-3) (u+2)/3$, which means that our inequality is satisfied, and equality is not actually possible. (We would need $d+2=(u+2)/3$, which would make $d$ even, hence not coprime to $p-1$.)

If we have $d=(u-1)/3$, then note that $p \geq 17$ since we have $d \geq 5$.
Furthermore, we have $-3 d \equiv 1 \pmod{u}$, so then
\[
\wt(3)+\wt(-3 d) = 3 + 1 = 4,
\]
and this is strictly less than $(u+2)/3$ since $p\geq 17$.

If we have $(u+2)/3 \leq d < u/2$, then $a=2$ and $r=u-2 d$, so then $\wt(a)+\wt(r)=2+u-2 d \leq (u+2)/3$ and we could only get equality if $d=(u+2)/3=(p+1)/3$, but this is impossible, since it would make $d$ even, hence not coprime to $u$.

We cannot have $d=u/2$, for then $\gcd(d,u)=u/2 > 1$.

We cannot have $d=(u/2)+1=(p+1)/2$, because we are assuming $d \not\equiv 1\pmod{(p-1)/2}$.

If $(u/2)+2 \leq d \leq 2(u-1)/3$, then write $d=(u/2)+e$, so that $2 \leq e \leq (u-4)/6$.
And write $u-d=2 e b + s$ with $0 \leq s < 2 e$, so that $b=\frac{(u/2)-e-s}{2 e}$.  Note that $b \leq \frac{u}{8} - \frac{1}{2}$.  Since $2 d \equiv 2 e \pmod{p-1}$, we have $-d(2 b+1)\equiv -2 b e -d \equiv s \pmod{p-1}$.
Thus
\[
\wt(2 b+1)+\wt(-d(2 b+1)) = \wt(2 b+1) + \wt(s) = 2 b + 1 + s,
\]
since $2 b+1 \leq \frac{u}{4} = \frac{p-1}{4} < p-1$ and $s < 2 e \leq (u-4)/3 = (p-5)/3 < p-1$.
So to prove our upper bound, it suffices to show that $2 b+1+s \leq (u+2)/3$, that is,
\[
\frac{(u/2)-s}{e} + s \leq \frac{u+2}{3},
\]
or equivalently,
\[
(e-1) s -1 \leq (2 e-3) \frac{u+2}{6}.
\]
This is in fact always satisfied strictly: since $s \leq 2 e-1$ and $2 \leq e \leq (u-4)/6$, we have $(e-1) s -1 \leq (e-1) (2 e-1) -1 = (2 e-3)e < (2 e-3) (u+2)/6$.

If we have $d \geq (2 u+1)/3$, then $a=1$ and $r=u-d$, so $\wt(a)+\wt(r)=1+u-d \leq (u+2)/3$, with equality possible only if $d=(2 u+1)/3=(2 p-1)/3$.

Thus we have proved our upper bound on $V_{\Fp,d}$ for all values of $d$, and have shown that equality can be achieved only if $d=3$ or $(2 p-1)/3$.
And Lemma \ref{Raphael} shows that in both of these cases we have $\gcd(d,p-1)=1$ and $V_{\Fp,d}=(p+1)/(3(p-1))$.
\end{proof}

\begin{proposition}\label{Samuel}
Let $p$ be a prime with $p \equiv 1 \pmod{4}$ and $p \geq 13$, let $F$ be the finite field of order $p^2$, and let $d$ be a positive integer with $d \equiv (p+1)/2 \pmod{p-1}$ and $\gcd(d,p^2-1)=1$.  Then $V_{F,d} \leq 1/2$.
\end{proposition}
\begin{proof}
We can assume that $d < p^2-1$ by replacing $d$ with the remainder one gets when one divides it by $p^2-1$.  Then we can write $d=\frac{p+1}{2}+ b (p-1)$ for some $b$ with $0 < b \leq p$.
(We cannot have $b=0$, for then we would have $\gcd(d,p^2-1) \geq \gcd(d,p+1) = \frac{p+1}{2} > 1$.)

We shall use Lemma \ref{Stanley} to bound $V_{F,d}$.
Let $\wt=\wt_{p,2}$ be the $p$-ary weight function on $\Z/(p^2-1)\Z$ as defined in \eqref{Walter}.
We use the convention that if $z \in \Z$, then $\wt(z)$ is a shorthand for $\wt(\bar{z})$, where $\bar{z} \in \Z/(p^2-1)\Z$ is the reduction modulo $p^2-1$ of $z$.

If $b=1$, then $d=(3 p-1)/2$.  Then note that $-(p+4) d = \frac{-3 p^2 -11 p +4}{2} \equiv \frac{p-11}{2} \cdot p \pmod{p^2-1}$, so that
\[
\wt(p+4) + \wt(-d(p+4)) = 5+ \frac{p-11}{2} = \frac{p-1}{2},
\]
and so $V_{F,d} \leq 1/2$ by Lemma \ref{Stanley}

If $1 < b < (p-1)/2$, then $d=\frac{(2 b+1) p - (2 b-1)}{2}$.  Then note that $-(p+2) d = \frac{-(2 b+1)p^2-(2 b+3) p + (4 b-2)}{2} \equiv \frac{p^2-(2 b+3)p + (2 b-4)}{2} \equiv  b-2 + p \cdot \frac {p-2 b-3}{2} \pmod{p^2-1}$, so that
\[
\wt(p+2)+\wt(-d(p+2)) = 3 + b -2 + \frac{p-2 b-3}{2} = \frac{p-1}{2},
\]
and so $V_{F,d} \leq 1/2$ by Lemma \ref{Stanley}.

If $b=(p-1)/2$, then $d=1 + p \cdot\frac{p-1}{2}$, and so $p d =  p + p^2 \cdot\frac{p-1}{2} \equiv p+\frac{p-1}{2} \equiv \frac{3 p-1}{2} \pmod{p^2-1}$, and we have shown (in the $b=1$ case) that $V_{F, p d} \leq 1/2$, so then $V_{F,d} \leq 1/2$ by Remark \ref{Penelope}.

We cannot have $b=(p+1)/2$, for then $\gcd(d,p^2-1) \geq \gcd(d,p+1) \geq (p+1)/2 > 1$.

If $(p+1)/2 < b \leq p$, then $d=\frac{(2 b+1) p - (2 b-1)}{2}$, so $-d =\frac{-(2 b+1) p + (2 b-1)}{2} \equiv \frac{(2 p-2 b-1) p + (2 b-3)}{2} \equiv (p-b) p + \frac{2 b-(p+3)}{2} \pmod{p^2-1}$, so that
\[
\wt(1)+\wt(-d) = 1 + p -b + \frac{2 b-(p+3)}{2} = \frac{p-1}{2},
\]
and so $V_{F,d} \leq 1/2$ by Lemma \ref{Stanley}.
\end{proof}

\begin{lemma}\label{Nancy}
Let $p$ be odd and let $d=(q+1)/2$.
Then $d\equiv 1 \pmod{p-1}$ if and only if $n$ is even.
We have $\gcd(d,q-1)=1$ if and only if $q\equiv 1 \pmod{4}$, in which case $V_{F,d} = n/2$.
\end{lemma}
\begin{proof}
First of all, $d=\frac{(p-1)}{2} (1+p+\cdots+p^{n-1}) + 1$, so $d \equiv 1 + n \cdot \frac{p-1}{2}\equiv 1 \pmod{p-1}$ if and only if $2 \mid n$.

Secondly, $\gcd(d,q-1)=\gcd(d,q-1-2 d)=\gcd(d,-2)$ and $2 \mid d$ if and only if $4 \mid q+1$, that is, if and only if $q\equiv 3 \pmod{4}$.

We assume that $\gcd(d,q-1)=1$ henceforth, and determine $V_{F,d}$ via Lemma \ref{Stanley}.  
Let $\wt=\wt_{p,n}$ be the $p$-ary weight function on $\Z/(p^n-1)\Z$ as defined in \eqref{Walter}.
Suppose $a$ is chosen among the nonzero $x \in \Z/(p^n-1)\Z$ that minimize $\wt(x)+\wt(-d x)$, and among such $x$, make sure that $a$ is one with $\wt(a)$ minimal.  Then we claim that $\wt(a) \leq 2$ because otherwise there is some nonzero $a'=a-p^j-p^k$ such that $\wt(a')=\wt(a)-2$, and then note that since $p$ is odd and $d-1=(q-1)/2$, we have $-d a' = -d a +(p^j+p^k)(q-1)/2 + p^j + p^k = -d a + p^j + p^k$, so that $\wt(-d a') \leq \wt(-d a)+2$, and so $\wt(a')+\wt(-d a') \leq \wt(a)+\wt(-d a)$.

If $\wt(a)=2$, say $a=p^j+p^k$, then $-d a =-(p^j+p^k)(q-1)/2 - p^j - p^k=-p^j-p^k=-a$, so then $\wt(a)+\wt(-d a)=n(p-1)$.  But if $\wt(a)=1$, then $a=p^i$ for some $i$, so that $\wt(-d a)=\wt(-d)=n(p-1)/2-1$, and so $\wt(a)+\wt(-d a)=n(p-1)/2$.  So
\[
\mins{a \in \Z/(p^n-1)\Z \\ a \not=0} \wt_{p,n}(a)+\wt_{p,n}(-d a) = \frac{n(p-1)}{2},
\]
and so Lemma \ref{Stanley} shows that $V_{F,d}=n/2$.
\end{proof}
Finally, we prove more results used in the proofs of Propositions \ref{Angela} and \ref{Bernard}.
\begin{lemma}\label{Peter}
Let $p\equiv 1 \pmod{3}$ and $d=(q+2)/3$.
Then $d\equiv 1 \pmod{p-1}$ if and only if $3\mid n$.
We have $\gcd(d,q-1)=1$ if and only if $q \pmod{9} \in \{1,4\}$, in which case $V_{F,d} = n/3$.
\end{lemma}
\begin{proof}
First of all, $d=\frac{p-1}{3} (1+p+\cdots+p^{n-1}) + 1$, so $d \equiv 1+ n \cdot \frac{p-1}{3} \equiv 1 \pmod{p-1}$ if and only if $3\mid n$.

Secondly, $\gcd(d,q-1)=\gcd(d,q-1-3 d)=\gcd(d,-3)$ and $3 \mid d$ if and only if $9 \mid q+2$, that is, if and only if $q\equiv 7 \pmod{9}$.

We assume that $\gcd(d,q-1)=1$ henceforth, and determine $V_{F,d}$ via Lemma \ref{Stanley}.
Let $\wt=\wt_{p,n}$ be the $p$-ary weight function on $\Z/(p^n-1)\Z$ as defined in \eqref{Walter}.
Suppose $a$ is chosen among the nonzero $x \in \Z/(p^n-1)\Z$ that minimize $\wt(x)+\wt(-d x)$, and among such $x$, make sure that $a$ is one with $\wt(a)$ minimal.  Then we claim that $\wt(a) \leq 3$ because otherwise there is some nonzero $a'=a-p^j-p^k-p^\ell$ such that $\wt(a')=\wt(a)-3$, and then note that since $p\equiv 1 \pmod{3}$ and $d-1=(q-1)/3$, we have $-d a' = -d a +(p^j+p^k+p^\ell)(q-1)/3+ p^j + p^k + p^\ell = -d a + p^j + p^k + p^\ell$, so that $\wt(-d a') \leq \wt(-d a)+3$, and so $\wt(a')+\wt(-d a') \leq \wt(a)+\wt(-d a)$.  So $\wt(a)=1, 2$, or $3$.

If $\wt(a)=3$, say $a=p^j+p^k+p^\ell$, then $-d a=-(p^j+p^k+p^\ell)(q-1)/3-p^j-p^k-p^\ell=-p^j-p^k-p^\ell$ since $p\equiv 1 \pmod{3}$.  So $-d a = -a$, and so $\wt(a)+\wt(-d a)=n(p-1)$.

If $\wt(a)=1$, say $a=p^i$, then $\wt(-d a)=\wt(-d)=\wt(2(q-1)/3-1)=2 n (p-1)/3-1$, and so $\wt(a)+\wt(-d a)=2 n (p-1) /3$.

If $\wt(a)=2$, say $a=p^j+p^k$, then $-d a=-(p^j+p^k)(q-1)/3-p^j-p^k$ and since $p \equiv 1 \pmod{3}$, this means $p^j+p^k-d a =(q-1)/3$.  So $n(p-1)/3=\wt(p^j+p^k-d a) \leq 2 + \wt(-d a)=\wt(a)+\wt(-d a)$.  And in fact, if we just pick $a=2$, then $\wt(-d a)=\wt((q-1)/3-2)=n(p-1)/3-2$, so that $\wt(a)+\wt(-d a)=n(p-1)/3$.  So
\[
\mins{a \in \Z/(p^n-1)\Z \\ a \not=0} \wt_{p,n}(a)+\wt_{p,n}(-d a) = \frac{n(p-1)}{3},
\]
and so Lemma \ref{Stanley} shows that $V_{F,d}=n/3$.
\end{proof}

\begin{lemma}\label{Paul}
Let $p\equiv 1 \pmod{3}$ and $d=(2 q+1)/3$.
Then $d\equiv 1 \pmod{p-1}$ if and only if $3\mid n$.
We have $\gcd(d,q-1)=1$ if and only if $q\pmod{9} \in \{1,7\}$, in which case $V_{F,d} = n/3$.
\end{lemma}
\begin{proof}
First of all, $d=\frac{2(p-1)}{3} (1+p+\cdots+p^{n-1}) + 1$, so $d \equiv 1 + 2 n \cdot \frac{p-1}{3} \equiv 1 \pmod{p-1}$ if and only if $3 \mid n$.

Secondly, since $d$ is odd we have $\gcd(d,q-1)=\gcd(d,2 q-2)=\gcd(d,2 q-2-3 d)=\gcd(d,-3)$ and $3 \mid d$ if and only if $9 \mid 2 q+1$, that is, if and only if $q\equiv 4\pmod{9}$.

We assume that $\gcd(d,q-1)=1$ henceforth, and determine $V_{F,d}$ via Lemma \ref{Stanley}.
Let $\wt=\wt_{p,n}$ be the $p$-ary weight function on $\Z/(p^n-1)\Z$ as defined in \eqref{Walter}.
Suppose $a$ is chosen among the nonzero $x \in \Z/(p^n-1)\Z$ that minimize $\wt(x)+\wt(-d x)$, and among such $x$, make sure that $a$ is one with $\wt(a)$ minimal.  Then we claim that $\wt(a) \leq 3$ because otherwise there is some nonzero $a'=a-p^j-p^k-p^\ell$ such that $\wt(a')=\wt(a)-3$, and then note that since $p\equiv 1 \pmod{3}$ and $d-1=2(q-1)/3$, we have $-d a' = -d a +2(p^j+p^k+p^\ell)(q-1)/3+ p^j + p^k + p^\ell = -d a + p^j + p^k + p^\ell$, so that $\wt(-d a') \leq \wt(-d a)+3$, and so $\wt(a')+\wt(-d a') \leq \wt(a)+\wt(-d a)$.  So $\wt(a)=1, 2$, or $3$.

If $\wt(a)=3$, say $a=p^j+p^k+p^\ell$, then $-d a=-2(p^j+p^k+p^\ell)(q-1)/3-p^j-p^k-p^\ell=-p^j-p^k-p^\ell$ since $p\equiv 1 \pmod{3}$.  So $-d a = -a$, and so $\wt(a)+\wt(-d a)=n(p-1)$.

If $\wt(a)=2$, say $a=p^j+p^k$, then $-d a=-2(p^j+p^k)(q-1)/3-p^j-p^k$ and since $p \equiv 1 \pmod{3}$, this means $p^j+p^k-d a =2(q-1)/3$.  So $2 n(p-1)/3=\wt(p^j+p^k-d a) \leq 2 + \wt(-d a)=\wt(a)+\wt(-d a)$. 

If $\wt(a)=1$, say $a=p^i$, then $\wt(-d a)=\wt(-d)=\wt((q-1)/3-1)=n (p-1)/3-1$, and so $\wt(a)+\wt(-d a)=n (p-1) /3$.  So
\[
\mins{a \in \Z/(p^n-1)\Z \\ a \not=0} \wt_{p,n}(a)+\wt_{p,n}(-d a) = \frac{n(p-1)}{3},
\]
and so Lemma \ref{Stanley} shows that $V_{F,d}=n/3$.
\end{proof}

\begin{lemma}\label{Raphael}
Let $p$ be odd and $p\equiv 2 \pmod{3}$.
Then $3\not\equiv 1 \pmod{p-1}$.
We have $\gcd(3,q-1)=1$ if and only if $q\equiv 2 \pmod{3}$, in which case the multiplicative inverse of $3$ modulo $q-1$ is $d=(2 q-1)/3$ and $V_{F,3}=V_{F,d}=n \cdot \frac{p+1}{3(p-1)}$, and $d \not\equiv 1 \pmod{p-1}$.
\end{lemma}
\begin{proof}
First of all, since $p$ is odd and congruent to $2$ modulo $3$, we have $p\geq 5$, so $3\not\equiv 1 \pmod{p-1}$.

Secondly, it is clear that $\gcd(3,q-1)=1$ if and only if $q\not\equiv 1 \pmod{3}$, and since $q$ is a power of $p$ (which is not $3$), this is true if and only if $q\equiv 2 \pmod{3}$.

We assume that $\gcd(3,q-1)=1$ henceforth and set $d=(2 q-1)/3$.
Then $3 d = 2 q-1 \equiv 1 \pmod{q-1}$, so $d$ is the multiplicative inverse of $3$ modulo $q-1$.  So $3 d \equiv 1 \pmod{p-1}$ and since $3\not\equiv 1 \pmod{p-1}$, this means that $d \not\equiv 1 \pmod{p-1}$.

Since $3$ and $d$ are inverses of each other modulo $q-1$, Remark \ref{Ignatius} tells us that $V_{F,3}=V_{F,d}$, so it remains to show that $V_{F,3}=n (p+1)/3$, which we now do using Lemma \ref{Stanley}.
Let $\wt=\wt_{p,n}$ be the $p$-ary weight function on $\Z/(p^n-1)\Z$ as defined in \eqref{Walter}.
Suppose $a$ is chosen among the nonzero $x \in \Z/(q-1)\Z$ that minimize $\wt(x)+\wt(-3 x)$, and among such $x$, make sure that $a$ is one with $\wt(a)$ minimal.
Write $a=a_0+a_1 p + \cdots +a_{n-1} p^{n-1}$ with $0 \leq a_i < p$ for each $i$, and at least one $a_i$ is nonzero, and at least one $a_i$ is not $p-1$.
If $a_i \geq (p+1)/3$ for some $i$, let $a'=a-p^i(p+1)/3$, so that $\wt(a')=\wt(a)-(p+1)/3$.
Note that $-3 a' = -3 a +(p+1) p^i$, so that $\wt(-3 a') \leq \wt(-3 a) + 2$, and thus $\wt(a')+\wt(-3 a') \leq \wt(a)+\wt(-3 a)-(p+1)/3+2 \leq \wt(a)+\wt(-3 a)$, and since $\wt(a') < \wt(a)$, this would contradict our choice of $a$ unless $a'=0$.
And if $a'=0$, then $a=p^i(p+1)/3$, so then $-3 a=-p^i(p+1)=-p^{i+1}-p^i$ and so $\wt(a)+\wt(-3 a)=(p+1)/3+n(p-1)-2 \geq n(p-1)$, and this would contradict the choice of $a$ since $\wt(1)+\wt(-3)=1+n(p-1)-3 < n(p-1)$.

So we must have $a_i \leq (p-2)/3$ for every $i$, and thus $3 a_i \leq p-2$ for every $i$.
So $3 a = 3 a_0 + 3 a_1 p + \cdots + 3 a_{n-1} p^{n-1}$ has $\wt(3 a)=3\wt(a)$, and so $\wt(a)+\wt(-3 a)=\wt(a)+n(p-1)-\wt(3 a)=n(p-1)-2 \wt(a)$.  And since $a_i \leq (p-2)/3$ for every $i$, we have $\wt(a)\leq n(p-2)/3$, and so $\wt(a)+\wt(-3 a) \geq  n (p+1)/3$, with equality if we let every $a_i=(p-2)/3$, that is, let $a=(p-2)(q-1)/(3(p-1))$.  So
\[
\mins{a \in \Z/(p^n-1)\Z \\ a \not=0} \wt_{p,n}(a)+\wt_{p,n}(-3 a) = \frac{n(p+1)}{3},
\]
and so Lemma \ref{Stanley} shows that $V_{F,3}=n(p+1)/(3(p-1))$.
\end{proof}

\section{Open Problems}\label{Roland}

Theorem \ref{Dorothy} gives us a universal bound $V_{F,d} \leq (2/3) [F\colon\Fp]$ when $d$ is nondegenerate over $F$.
When $[F\colon\Fp]$ is a multiple of $3$, Lemma \ref{Gerald} says that this universal bound is always attained for some $d$.
Remark \ref{Michael} furnishes a stronger bound of $V_{F,d} \leq (1/2) [F\colon\Fp]$ when $[F\colon\Fp]=2^s$ with $s \geq 1$ and $d$ is nondegenerate over $F$, and for each such $F$, this stronger bound is always attained for some $d$ by Lemma \ref{Nora} (except when $F=\F_4$, over which there is no nondegenerate $d$).
When $F$ is a prime field $\Fp$ and $d$ is nondegenerate over $F$ (which requires $p \geq 5$), Theorem \ref{Natasha} gives upper bounds of $V_{F,d} \leq 1/2$ when $p\equiv 1\pmod{4}$, and $V_{F,d} \leq \ceil{(p-1)/3}/(p-1)$ when $p\equiv 3 \pmod{4}$, and Lemmata \ref{Nancy}--\ref{Raphael} show that these bounds are always met for some $d$ in every such $F$.

So it remains to determine precisely how high $V_{F,d}$ can be when $[F\colon\Fp]$ is neither a power of $2$ nor a multiple of $3$.
If $\ell$ is the least odd prime divisor of $[F\colon\Fp]$, then Lemma \ref{Gerald} shows that there is some $d$ such that $V_{F,d}=\frac{\ell+1}{2\ell} [F\colon\Fp]$.  We are unaware of any pair $(F,d)$ where this value is exceeded.
Thus we make the following conjecture that the value of $V_{F,d}$ observed in Lemma \ref{Gerald} is in fact the upper bound.
\begin{conjecture}[Upper Bound Conjecture]\label{Simon}
Suppose that $[F:\Fp]$ is not a power of $2$, and let $\ell$ be the least odd prime divisor of $[F:\Fp]$.
Let $d$ be a positive integer with $\gcd(d,q-1)=1$ that is not degenerate over $F$.
Then $V_{F,d} \leq \frac{\ell+1}{2\ell} [F\colon\Fp]$.
\end{conjecture}
This bound coincides with the universal bound $V_{F,d} \leq (2/3) [F\colon\Fp]$ when $\ell=3$, but is often stronger when $\ell > 3$.
Computer checks have verified Conjecture \ref{Simon} for all fields $F$ of order less than $10^{13}$.

Again consider our universal bound $V_{F,d} \leq (2/3) [F\colon\Fp]$ for $d$ nondegenerate over $F$.
It is interesting that the proof (in Section \ref{Victor}) does not use Stickelberger's Theorem (which underlies Lemma \ref{Stanley}).  Attempts to prove the universal bound directly with Stickelberger's Theorem lead to an interesting conjecture in elementary number theory that, if true, would provide an alternative proof for the universal bound.  To state the conjecture, recall that if $t$ is an integer with $t \geq 2$ and $n$ is a positive integer, then we define the standard $t$-ary expansion of an $a \in \Z/(t^n-1)\Z$ to be the expression
\[
a=a_0 t^0 + a_1 t^1 + \cdots + a_{n-1} t^{n-1},
\]
where the powers of $t$ are elements of $\Z/(t^n-1)\Z$ and $a_0,\ldots,a_{n-1}$ are elements of $\Z$ with $0 \leq a_i < t$ for every $i$, and where we insist that $a_0=\cdots=a_n=0$ when $a=0$ (to make the $a_i$'s uniquely defined).
If $b \in \Z/(t^n-1)\Z$ has standard $t$-ary expansion $b=b_0+b_1 t+\cdots + b_{n-1} t^{n-1}$, then we say that {\it $b$ covers $a$} and write $a \preceq b$ to indicate that $a_i \leq b_i$ for every $i$.  If $a\preceq b$ and $a\not=b$, we say that {\it $b$ strictly covers $a$} and write $a \prec b$.
\begin{conjecture}[Covering Conjecture]\label{Ilse}
Let $t$ be an integer with $t\geq 2$ and let $n$ and $d$ be positive integers such that $d$ modulo $t^n-1$ is neither zero nor a power of $t$.  Then there exist nonzero $a, b \in \Z/(t^n-1)\Z$ such that $a \prec b$ and $d b \prec d a$.
\end{conjecture}
To see that this conjecture would provide an alternative proof of our universal bound (Theorem \ref{Dorothy}), let $d$ be a positive integer coprime to $p^n-1$ and nondegenerate over $F$.  Then Conjecture \ref{Ilse} would show that there are nonzero $a, b \in \Z/(p^n-1)\Z$ such that $a \prec b$ and $d b \prec d a$.  Let $\wt=\wt_{p,n}$ be the $p$-ary weight function for $\Z/(p^n-1)\Z$ as defined in \eqref{Walter}.  Now we shall use Stickelberger's Theorem via Lemma \ref{Stanley}: since $a$, $-b$, and $b-a$ are all nonzero elements of $\Z/(p^n-1)\Z$, we will obtain our universal bound $V_{F,d} \leq 2 n/3$ if we can prove that at least one of $\alpha=\wt(-d a)+\wt(a)$, $\beta=\wt(d b)+\wt(-b)$ or $\gamma=\wt(d a-d b)+\wt(b-a)$ is less than or equal to $(2/3)n(p-1)$.  Since $b$ covers $a$, when we add $a$ and $b-a$ to obtain $b$, there are no carries (in base $p$ representation), so $\wt(a)+\wt(b-a)=\wt(b)$.  And thus $\wt(a)+\wt(b-a)+\wt(-b)=\wt(b)+\wt(-b)=n(p-1)$.  Similarly, we have $\wt(d b)+\wt(d a - d b)+\wt(-d a)=n(p-1)$.
Thus $\alpha+\beta+\gamma=2 n(p-1)$ and so at least one of the three summands is less than or equal to $(2/3)n(p-1)$.

We have considerable evidence for the truth of Conjecture \ref{Ilse}.
To see this, we first provide some observations and partial proofs.
The first observation shows that one only needs to check the conjecture for bases $t$ that are not powers of smaller integers.
\begin{remark}
If Conjecture \ref{Ilse} is true when $t=t_1$ and $n=n_1$, then it is also true when $t=t_1^k$ and $n=n_1/k$ for any positive divisor $k$ of $n_1$.  For if $x \prec y$ when $x$ and $y$ are expressed in standard $t_1$-ary expansions, then $x \prec y$ when they are expressed in standard $(t_1^k)$-ary expansions.
\end{remark}
The second observation is that Conjecture \ref{Ilse} is trivial when $d$ is not coprime to $t^n-1$.
\begin{lemma}\label{Vincent}
Let $t$, $n$, and $d$ be positive integers with $t\geq 2$ and $1 < \gcd(d,t^n-1) < t^n-1$.  Then there exist nonzero $a, b \in \Z/(t^n-1)\Z$ such that $a \prec b$ and $d b \prec d a$.
\end{lemma}
\begin{proof}
Let $e=(t^n-1)/\gcd(d,t^n-1)$, so that $1 < e=\gcd(e,t^n-1) < t^n-1$ and $d e \equiv 0 \pmod{t^n-1}$.
Let $\bar{e} \in \Z/(t^n-1)\Z$ be the reduction of $e$ modulo $t^n-1$.
This $e$ is neither zero nor a power of $t$ modulo $t^n-1$, so there exists some $k$ such that $t^k \prec \bar{e}$.  And $d t^k$ is a nonzero element of $\Z/(t^n-1)\Z$ because $\gcd(d t^k,t^n-1)=\gcd(d,t^n-1) < t^n-1$.  Thus $d \bar{e} = 0 \prec d t^k$.
\end{proof}
There is a useful principle for lifting instances of covering to higher moduli.
\begin{lemma}\label{Leonard}
Let $t$, $m$, $n$, and $d$ be positive integers with $t \geq 2$ and $m \mid n$.  Suppose that there are nonzero $a, b \in \Z/(t^m-1)\Z$ such that $a \prec b$ and $d b \prec d a$ in $\Z/(t^m-1)\Z$.  Then there are nonzero $A, B \in \Z/(t^n-1)\Z$ such that $A \prec B$ and $d B \prec d A$ in $\Z/(t^n-1)\Z$.
\end{lemma}
\begin{proof}
Let $g$ be the integer $(t^n-1)/(t^m-1)$ and let $A, B$ be the unique elements of $\Z/(t^n-1)\Z$ given by $A=g a$ and $B=g b$.  (The fact that $a$ and $b$ are well defined modulo $t^m-1$ makes $g a$ and $g b$ well defined modulo $t^n-1$.)  Then the $t$-ary expansion of $A$ in $\Z/(t^n-1)\Z$ is just the $(n/m)$-fold repetition of the $t$-ary expansion of $a$ in $\Z/(t^m-1)\Z$, and similarly with $B$ relative to $b$, $d A$ relative to $d a$, and $d B$ relative to $d b$.  So $A \prec B$ and $d B \prec d A$.
\end{proof}
When $d$ is invertible modulo $t^n-1$, the conclusion of Conjecture \ref{Ilse} can often be deduced from direct examination of the standard $t$-ary expansions of $d \pmod{t^n-1}$ and its multiplicative inverse.
\begin{lemma}\label{Colin}
Let $t$ be an integer with $t \geq 2$ and let $n$ and $d$ be positive integers with with $\gcd(d,t^n-1)=1$ and $d$ not a power of $t$ modulo $t^n-1$.  Let $\bar{d} \in \Z/(t^n-1)\Z$ be the reduction of $d$ modulo $t^n-1$, and let $\bar{e}$ be the multiplicative inverse of $\bar{d}$.  Suppose that $\bar{d}$ and $\bar{e}$ have standard $t$-ary expansions $\bar{d}=d_0+ d_1 t+\cdots+d_{n-1} t^{n-1}$ and $\bar{e}=e_0+e_1 t + \cdots + e_{n-1} t^{n-1}$, respectively, and there exist some $j,k \in \Z/n\Z$ with $j+k\equiv 0 \pmod{n}$ such that $d_j\not=0$ and $e_k \not=0$.  Then $t^k \prec \bar{e}$ and $d \bar{e} \prec d t^k$.
\end{lemma}
\begin{proof}
It is clear that $t^k \preceq \bar{e}$ because $e_k\not=0$.  And in fact $\bar{e}\not=t^k$, because then its inverse would be $t^{n-k}$, but we were given that $d$ is not a power of $t$ modulo $t^n-1$.  So $t^k \prec\bar{e}$.

The $t$-ary digits of $d t^k$ are obtained by cyclically shifting those of $\bar{d}$, so that the $j$th digit of $\bar{d}$ (which is nonzero) becomes the $0$th digit of $d t^k$.  Thus $d\bar{e} = 1 \preceq d t^k$.  And $1\not=d t^k$ because that would make $\bar{d}=t^{n-k}$, and $d$ is not a power of $t$ modulo $t^n-1$.
\end{proof}
These principles allow us to prove that Conjecture \ref{Ilse} becomes true if we add $d\not\equiv 1\pmod{t-1}$ as an hypothesis.
\begin{lemma}\label{Nina}
Let $t$, $n$, and $d$ be positive integers with $t \geq 2$, $d\not\equiv 0 \pmod{t^n-1}$, and $d\not\equiv 1 \pmod{t-1}$.  Then there exist nonzero $a,b \in \Z/(t^n-1)\Z$ such that $a \prec b$ and $d b \prec d a$.
\end{lemma}
\begin{proof}
We may assume $\gcd(d,t^n-1) = 1$ because otherwise Lemma \ref{Vincent} guarantees our result.
Then the hypotheses of Lemma \ref{Colin} are clearly satisfied when $n=1$, thus establishing our result in that case.  Then the cases with $n > 1$ can be deduced from the $n=1$ case along with Lemma \ref{Leonard}.
\end{proof}
And we also can prove Conjecture \ref{Ilse} when $n \leq 4$.
\begin{lemma}
Let $t$ and $n$ be integers with $t \geq 2$ and $n \leq 4$, and let $d$ be a positive integer such that $d$ is neither zero nor a power of $t$ modulo $t^n-1$.  Then there exist nonzero $a,b \in \Z/(t^n-1)\Z$ such that $a \prec b$ and $d b \prec d a$.
\end{lemma}
\begin{proof}
We may assume that $\gcd(d,t^n-1)=1$, because otherwise Lemma \ref{Vincent} gives us our result immediately.
Without loss of generality we assume that $d < t^n-1$ by replacing it with its remainder upon division by $t^n-1$, and write $d=d_0+\cdots+d_{n-1} t^{n-1}$ with $0 \leq d_i < t$ for each $d_i$.
In view of Lemma \ref{Nina}, we may assume that $d \equiv 1 \pmod{t-1}$, and so $d_0+\cdots+d_n \equiv 1 \pmod{t-1}$.  Since $d$ is not a power of $t$ modulo $t^n-1$, this means that $d_0+\cdots+d_n \geq t$.  Thus at least two of the $d_i$'s are nonzero.
And if we let $e$ be the integer with $0 \leq e < t^n-1$ such that $d e \equiv 1 \pmod{t^n-1}$, and write $e=e_0+\cdots+e_{n-1} t^{n-1}$ with $0 \leq e_i < t$ for each $e_i$, then $e$ is not a power of $t$ modulo $t^n-1$ and $e \equiv 1\pmod{t-1}$, so that at least two of the $e_i$'s are nonzero.
If $n \leq 3$, this means that the hypotheses of Lemma \ref{Colin} are satisfied, and so our conclusion follows.

So we may assume $n=4$ henceforth.
If $d_0+d_1+d_2+d_3 > t$, then the fact that $d_0+d_1+d_2+d_4 \equiv 1 \pmod{t-1}$ forces $d_0+d_1+d_2+d_3 \geq 2 t-1$, which makes at least three of the $d_i$'s nonzero, and since we already know that at least two of the $e_i$'s are nonzero, the hypotheses of Lemma \ref{Colin} are satisfied, and so our conclusion follows.
Similarly, if $e_0+e_1+e_2+e_3 > t$, our conclusion will follow from Lemma \ref{Colin}.
So we may assume $d_0+d_1+d_2+d_3=e_0+e_1+e_2+e_3=t$ henceforth.

We cannot have $d \equiv 0 \pmod{t^2-1}$, for we are assuming that $\gcd(d,t^4-1)=1$, which would force $t^2-1=1$, which is absurd.
If $d$ is not a power of $t$ modulo $t^2-1$, then by the $n=2$ case of this lemma (which has already been established), we have some $a, b \in \Z/(t^2-1)\Z$ such that $a \prec b$ and $d b \prec d a$ in $\Z/(t^2-1)\Z$.  Then Lemma \ref{Leonard} furnishes $A, B \in \Z/(t^4-1)\Z$ with $A \prec B$ and $d B \prec d A$ in $\Z/(t^4-1)\Z$, and we are done.

If $d$ is a power of $t$ modulo $t^2-1$, say $d \equiv t^k \pmod{t^2-1}$, then $e \equiv t^k \pmod{t^2-1}$ also.
Since $d \equiv (d_3+d_1) t + (d_2+d_0) \pmod{t^2-1}$ is a power of $t$ and $d_0+d_1+d_2+d_3=t$, we know that $(d_3+d_1) t + (d_2+d_0)$ is not the standard $t$-ary expansion of $d$ modulo $t^2-1$.  So we must have $\{d_3+d_1,d_2+d_0\}=\{0,t\}$, and similarly $\{e_3+e_1,e_2+e_0\}=\{0,t\}$.  Furthermore, since $d \equiv e \pmod{t^2-1}$, we either have $d_3+d_1=e_3+e_1=t$ or $d_2+d_0=e_2+e_0=t$.  Since all $d_i$'s and $e_i$'s are less than $t$, we know that either $d_3,d_1,e_3,e_1$ are all nonzero or else $d_2,d_0,e_2,e_0$ are all nonzero, and so the hypotheses of Lemma \ref{Colin} are satisfied, and our conclusion follows.
\end{proof}
In addition to these partial proofs, computer checks also verify Conjecture \ref{Ilse} for all $t^n$ less than $3\cdot 10^9$.

\section*{Acknowledgements}

The first author thanks Pascal V\'eron and Alicia Weng for stimulating discussions on some of the topics in this paper.  The authors thank an anonymous reviewer for helpful comments and corrections that improved the paper.

\end{document}